\documentclass[reqno]{amsart}
\usepackage{amsmath}
\usepackage{amssymb}
\usepackage{amsthm}
\usepackage{mathrsfs}
\usepackage{ulem}
\usepackage{url}
\usepackage{cancel}
\newtheorem{definition}{Definition}[section]
\newtheorem{theorem}[definition]{Theorem}
\newtheorem{lemma}[definition]{Lemma}

\newtheorem{cor}[definition]{Corollary}

\newtheorem{example}[definition]{Example}

\newtheorem{remark}[definition]{Remark}

\usepackage{color}

\newcommand{\M}{\mathrm{Mult}(H_k)}
\newcommand{\C}{\mathbb{C}}
\newcommand{\D}{\mathbb{D}}

\DeclareRobustCommand{\erase}{\bgroup\markoverwith{\textcolor{red}{\rule[.5ex]{2pt}{0.4pt}}}\ULon}

\numberwithin{equation}{section}

\begin{document}

\title[An RKHS approach to the indefinite Schwartz--Pick inequality]{An RKHS approach to the indefinite Schwarz--Pick inequality on the bidisk}
\author{Kenta Kojin}
\address{
Graduate School of Mathematics, Nagoya University, 
Furocho, Chikusaku, Nagoya, 464-8602, Japan
}
\email{m20016y@math.nagoya-u.ac.jp}
\date{\today}


\begin{abstract}
In this short note, we will give a generalization of the indefinite Schwarz--Pick inequality due to Seto \cite{Set}. Our approach is based on a connection between complex geometry and the geometry of reproducing kernel Hilbert spaces, which was crucially used in our previous work \cite{Koj}.
\end{abstract}

\maketitle

\section{Introduction}

Let $\D$ denote the open unit disk in the complex plane $\C$. 
For any two points $(z_1,z_2)$ and $(w_1,w_2)$ in $\D^2$, 
we define
\begin{equation*}
\rho((z_1, z_2), (w_1, w_2))=\sqrt{\left|\dfrac{z_1-w_1}{1-\overline{w_1}z_1}\right|^2+\left|\dfrac{z_2-w_2}{1-\overline{w_2}z_2}\right|^2
-\left|\dfrac{z_1-w_1}{1-\overline{w_1}z_1}\cdot \dfrac{z_2-w_2}{1-\overline{w_2}z_2}\right|^2}. 
\end{equation*}
This is a distance on $\D^2$ invariant under the automorphism group of $\D^2$. Seto \cite[Theorem 4.1]{Set} proved the following Schwarz--Pick type inequality for $\rho$:

\begin{theorem}\label{theorem:Seto}
    If $F:\D^2\rightarrow\D^2$ is a holomorphic map on $\D^2$, then $F$ satisfies
\begin{equation*}
\rho(F(z_1, z_2), F(w_1, w_2))\le \sqrt{2}\rho((z_1, z_2), (w_1, w_2))
\end{equation*}
for all $(z_1, z_2), (w_1, w_2)\in\D^2$.
\end{theorem}

He used the theory of Hilbert modules in the Hardy space over the bidisk to prove this result. 
In this note, we will give a generalization of Theorem \ref{theorem:Seto} as an application of the geometry of reproducing kernel Hilbert spaces (RKHS for short) with ``the two-point Nevanlinna--Pick property" (see Definition \ref{def:two-point Pick}). This property means that an RKHS enjoys a two point analog of the celebrated Nevanlinna--Pick interpolation theorem. 
It is well known that the original Schwarz--Pick inequality is equivalent to the two-point Nevanlinna--Pick interpolation theorem \cite[Lemma 0.3]{AMp}. 
Therefore, our approach also seems to be natural.

We will prove our main theorem in Section 3 (see Theorem \ref{theorem:1}). In section 2, we will recall a relation between a distance derived from complex geometry and a distance derived from RKHS theory, which played an important role in \cite{Koj}.






\section{Preliminaries}
Let $H_k$ be an RKHS of holomorphic functions on a domain $X\subset\C^d$  with reproducing kernel $k:X\times X\rightarrow\C$. We will always assume that 
$H_k$ is irreducible in the following sense: For any two points $x,y\in X$, the reproducing kernel $k$ satisfies $k(x,y)\ne 0$, and if $x\ne y$, then $k(\cdot, x)$ and $k(\cdot, y)$ are linearly independent in $H_k$. A function $\phi:X\rightarrow\C$ is called a multiplier of $H_k$ if $\phi f\in H_k$ for all $f\in H_k$. In this case, the multiplication operator $M_{\phi}:H_k\rightarrow H_k$ defined by $M_{\phi}f:=\phi f$ is bounded by the closed graph theorem, and the multiplier norm $\|\phi\|_{\M}$ is defined to be the operator norm of $M_{\phi}$. The multiplier algebra of $H_k$, denoted by $\M$, is the algebra consisting of multipliers of $H_k$. 
Since $k(x,y)\ne 0$ for all $x,y\in X$, every $\phi\in\M$ falls in $H^{\infty}(X)$ and satisfies $\|\phi\|_{\M}\ge\|\phi\|_{\infty}:=\sup\{|\phi(x)|\;|\;x\in X\}$. Here, $H^{\infty}(X)$ is the algebra of all bounded holomorphic functions on $X$.

For an RKHS $H_k$ of holomorphic functions on a domain $X\subset\C^d$, we define a pseudo-distance on $X$ by

\begin{equation*}
    d_k(x,y):=\sqrt{1-\frac{|k(x,y)|^2}{k(x,x)k(y,y)}}\;\;\;\;(x,y\in X).
\end{equation*}

It is well known, see e.g., \cite[Theorem 6.22]{ARSW2019}, that every contractive multiplier $\phi\in\M$ 
satisfies
\begin{equation*}
d_{\D}(\phi(x),\phi(y))\le d_k(x,y)\;\;\;\;(x,y\in X).
\end{equation*}
Here, $d_{\D}$ is the pseudo-hyperbolic distance 
\begin{equation*}
    d_{\D}(z,w):=\left|\frac{z-w}{1-\overline{w}z}\right|\;\;\;\;(z,w\in\D).
\end{equation*}

We define another pseudo-distance on $X$,
called the M$\ddot{\mathrm{o}}$bius distance, 
by
\begin{equation*}
d_X(x,y):=\sup_{\phi\in H(X,\D)}d_{\D}(\phi(x), \phi(y))\;\;\;\;(x,y\in X).
\end{equation*}
Here, $H(X, \D)$ is the set of all holomorphic functions from $X$ into $\D$.
We can prove the following Schwarz--Pick type inequality by the definition of the M$\ddot{\mathrm{o}}$bius distance:
\begin{lemma}\label{lemma:1}
Let $X\subset\C^d$ and $Y\subset\C^e$ be domains. If $F:X\rightarrow Y$ is a holomorphic map, then we have
\begin{equation*}
d_Y(F(x),F(y))\le d_X(x,y)
\end{equation*}
for all $x,y\in X$.
\end{lemma}

Moreover, $d_X$ enjoys the following product property:

\begin{lemma}[$\mbox{\cite[Theorem 4.9.1]{Kob}}$]\label{lemma:3}
Let $X\subset\C^d$ be a domain. Then, we have
\begin{equation*}
d_{X\times X}((x_1, x_2), (y_1, y_2))=\max\{d_X(x_1, y_1), d_X(x_2, y_2)\}
\end{equation*}
for all $(x_1, x_2), (y_1, y_2)\in X\times X$.
\end{lemma}

We remark that the Carath\'{e}odory distance on $X$, denoted by $c_X$, is given by 
\begin{equation*}
c_X(x,y):=\tanh^{-1}(d_X(x,y))\;\;\;\;(x,y\in X).
\end{equation*}

If an RKHS $H_k$ has the following two-point Nevanlinna--Pick property and enjoys $\M=H^{\infty}$ isometrically, then we can connect $d_k$ and $d_X$.



\begin{definition}\label{def:two-point Pick}
We say that an RKHS $H_k$ of holomorphic functions on a domain $X\subset \C^d$ has the  {\bf two-point Nevanlinna--Pick property} if whenever $x_1, x_2$ are points in $X$ and $w_1,w_2$ are complex numbers such that
\begin{equation*}
\begin{bmatrix}
    (1-|w_1|^2)k(x_1, x_1)&(1-w_1\overline{w_2})k(x_1, x_2)\\
    (1-w_2\overline{w_1})k(x_2, x_1)&(1-|w_2|^2)k(x_2, x_2)
\end{bmatrix}
\end{equation*}
is positive semidefinite, then there exists a $\phi\in\M$ such that $\phi(x_i)=w_i$ for $i=1,2$ and $\|\phi\|_{\M}\le 1$. 
\end{definition}

Here is the precise statement of the relation between $d_k$ and $d_X$. 

\begin{lemma}[\mbox{\cite[Lemma 5.2]{Koj}}]\label{lemma:2}
Let $H_k$ be an RKHS of holomorphic functions on a domain $X\subset\C^d$. If $H_k$ has the two-point Nevanlinna--Pick property and enjoys $\M=H^{\infty}(X)$ isometrically, then 
\begin{equation*}
d_X(x,y)=d_k(x,y)
\end{equation*}
 holds for every $x,y\in X$.
\end{lemma}

Of course, the Hardy space $H^2$ is a prototypical example of such an RKHS \cite{AMp}.  


\begin{remark}
\upshape Lemma \ref{lemma:2} still holds if $X$ is a reduced complex space \cite{Koj}. Here, reduced complex spaces are generalizations of complex manifolds. 
For simplicity, we assume that $X$ is a domain in $\C^d$. 
\end{remark}

\section{Main Theorem}

Let $H_k$ be an RKHS of holomorphic functions on a domain $X\subset\C^d$. 
Then $H_k\otimes H_k$ is an RKHS of holomorphic functions on $X\times X\subset\C^d\times \C^d$ and its reproducing kernel is given by 
\begin{equation*}
k^{\otimes2}((x_1, x_2), (y_1, y_2))=k(x_1, y_1)k(x_2, y_2)\;\;\;\;((x_1,x_2), (y_1,y_2)\in X\times X)
\end{equation*}
(see \cite[Section 5.5]{PR}). 
Moreover, we can prove that
\begin{align*}
    d_{k^{\otimes 2}}((x_1, x_2),(y_1, y_2))&=\sqrt{d_k(x_1, y_1)^2+d_k(x_2, y_2)^2-d_k(x_1, y_1)^2d_k(x_2, y_2)^2}\notag 
\end{align*}
(cf. \cite[Section 5.3]{ARSW2011}). 

The pseudo-distance $d_{k^{\otimes 2}}$ separates points of $X\times X$. To prove this,
we recall the assumption that $k(\cdot, x)$ and $k(\cdot, y)$ are linearly independent in $H_k$ for each distinct points $x, y\in X$. This is equivalent to the fact that the $2\times 2$ matrix
\begin{equation*}
\begin{bmatrix}
k(x,x)&k(x,y)\\
k(y,x)&k(y,y)
\end{bmatrix}
\end{equation*}
is strictly positive. Thus, the Schur product theorem implies that the $2\times 2$ matrix 
\begin{align*}
&\begin{bmatrix}
    k^{\otimes2}((x_1, x_2), (x_1, x_2))&k^{\otimes2}((x_1, x_2), (y_1, y_2))\\
    k^{\otimes2}((y_1, y_2), (x_1, x_2))&k^{\otimes2}((y_1, y_2), (y_1, y_2))
\end{bmatrix}\\
&=
\begin{bmatrix}
    k(x_1, x_1)k(x_2, x_2)&k(x_1, y_1)k(x_2, y_2)\\
    k(y_1, x_1)k(y_2, x_2)&k(y_1, y_1)k(y_2, y_2)
\end{bmatrix}
\end{align*}
is strictly positive for each distinct $(x_1, x_2), (y_1, y_2)\in X\times X$. This implies that $k^{\otimes 2}((\cdot, \cdot), (x_1, x_2))$ and $k^{\otimes 2}((\cdot, \cdot), (y_1, y_2))$ are linearly independent in $H_k\otimes H_k$. Therefore, $d_{k^{\otimes 2}}$ separates points of $X\times X$ by \cite[Lemma 9.9]{AMp}.

\begin{example}\label{example:3-1}
\upshape The reproducing kernel of the Hardy space $H^2$ is the Szeg\H{o} kernel
\begin{equation*}
k^S(z,w)=\frac{1}{1-\overline{w}z}\;\;\;\;(z,w\in\D).
\end{equation*}
Obviously, we have
\begin{equation*}
    d_{k^S}(z,w)=\left|\frac{z-w}{1-\overline{w}z}\right|\;(=d_{\D}(z,w))\;\;\;\;(z,w\in\D).
\end{equation*}
Therefore, in this case, we get
\begin{align*}
d_{(k^S)^{\otimes 2}}((z_1, z_2),(w_1, w_2))&=\rho((z_1, z_2),(w_1, w_2))\\
&=\sqrt{\left|\dfrac{z_1-w_1}{1-\overline{w_1}z_1}\right|^2+\left|\dfrac{z_2-w_2}{1-\overline{w_2}z_2}\right|^2
-\left|\dfrac{z_1-w_1}{1-\overline{w_1}z_1}\cdot \dfrac{z_2-w_2}{1-\overline{w_2}z_2}\right|^2}. 
\end{align*}
\end{example}

We can construct another RKHS from $H_k$. By the Schur product theorem, $k^n(x,y):=k(x,y)^n$ ($n\in \mathbb{N}$) is a reproducing kernel on $X$. Thus, $k^n$ defines an RKHS $H_{k^n}$ of holomorphic functions on $X$.

\begin{example}
\upshape The reproducing kernel of the Bergman space on $\D$ is 
\begin{equation*}
(k^S)^2(z,w)=\frac{1}{(1-\overline{w}z)^2}=(k^S)^{\otimes 2}((z,z), (w,w))\;\;\;\;(z,w\in \D).
\end{equation*}
Thus, we obtain
\begin{align*}
d_{(k^S)^2}(z,w)&=\sqrt{\left|\dfrac{z-w}{1-\overline{w}z}\right|^2+\left|\dfrac{z-w}{1-\overline{w}z}\right|^2
-\left|\dfrac{z-w}{1-\overline{w}z}\cdot \dfrac{z-w}{1-\overline{w}z}\right|^2}\\
&=\sqrt{2\left|\dfrac{z-w}{1-\overline{w}z}\right|^2-\left|\dfrac{z-w}{1-\overline{w}z}\right|^4}\;\;\;\;(z,w\in\D).
\end{align*}
We remark that the Bergman space does not have the two-point Nevanlinna--Pick property \cite[Example 5.17]{AMp}.
\end{example}

Here is a relation between $d_k$ and $d_{k^n}$. We will use this to prove our main theorem. This immediately follows from the definition of these distances.

\begin{lemma}\label{lemma:power}
Let $x, y, x', y'$ be points in $X$. Then, they satisfy
\begin{equation*}
d_k(x', y')\le d_k(x,y)
\end{equation*}
if and only if they satisfy
\begin{equation*}
d_{k^n}(x', y')\le d_{k^n}(x,y).
\end{equation*}
\end{lemma}

We also need the following lemma:

\begin{lemma}\label{lemma:SP}
Let $\zeta, \lambda, z, w$ be points in $\D$ such that
\begin{equation*}
    \max\{|\zeta|,|\lambda|\}\le \max\{|z|,|w|\}.
\end{equation*}
Then, these four points must satisfy
\begin{equation*}
    \sqrt{|\zeta|^2+|\lambda|^2-|\zeta\lambda|^2}\le \sqrt{2}\sqrt{|z|^2+|w|^2-|zw|^2}.
\end{equation*}
\end{lemma}

\begin{proof}
Since $0\le 1-|\lambda|^2\le 1$, we have
\begin{align*}
    \sqrt{|\zeta|^2+|\lambda|^2-|\zeta\lambda|^2}&=\sqrt{|\zeta|^2(1-|\lambda|^2)+|\lambda|^2}\le \sqrt{2}\max\{|\zeta|,|\lambda|\}\\
    &\le \sqrt{2}\max\{|z|,|w|\}\;\;\;\;(\mbox{by the hypothesis}).
\end{align*}
It is obvious that  
\begin{equation*}
    \max\{|z|,|w|\}\le \sqrt{|z|^2+|w|^2-|zw|^2}
\end{equation*}
holds. Therefore, we obtain
\begin{equation*}
    \sqrt{|\zeta|^2+|\lambda|^2-|\zeta\lambda|^2}\le \sqrt{2}\sqrt{|z|^2+|w|^2-|zw|^2}.
\end{equation*}
\end{proof}

Here is our main theorem. This is a generalization of Theorem \ref{theorem:Seto} 

\begin{theorem}\label{theorem:1}
Let $H_k$ be an RKHS of holomorphic functions on a domain $X\subset\C^d$. Suppose $H_k$ has the two-point Nevanlinna--Pick property and enjoys $\M=H^{\infty}(X)$ isometrically. If $F=(f_1, f_2):X\times X\rightarrow X\times X$ is a holomorphic map on $X\times X$, then $F$ satisfies
\begin{equation*}
d_{(k^n)^{\otimes 2}}(F(x_1, x_2), F(y_1, y_2))\le\sqrt{2}d_{(k^n)^{\otimes 2}}((x_1, x_2),(y_1, y_2))
\end{equation*}
for all $(x_1, x_2)$, $(y_1, y_2)$ in $X\times X$ and $n\in\mathbb{N}$. 
\end{theorem}

We remark that the constant $\sqrt{2}$ is the best possible one (see \cite[Remark 3.3]{DKS}). 

\begin{proof}
By Lemma \ref{lemma:1} and Lemma \ref{lemma:3}, we have
\begin{align*}
&\max\{d_X(f_1(x_1, x_2), f_1(y_1, y_2)), d_X(f_2(x_1, x_2), f_2(y_1, y_2))\}\\
&=d_{X\times X}(F(x_1, x_2), F(y_1, y_2))\;\;\;\;(\mbox{by Lemma \ref{lemma:3}})\\
&\le d_{X\times X}((x_1, x_2), (y_1, y_2))\;\;\;\;(\mbox{by Lemma \ref{lemma:1}})\\
&=\max\{d_X(x_1, y_1), d_X(x_2, y_2)\}\;\;\;\;(\mbox{by Lemma \ref{lemma:3}}).
\end{align*}
Thus, Lemma \ref{lemma:2} tells us that 
\begin{align*}
    &\max\{d_k(f_1(x_1, x_2), f_1(y_1, y_2)), d_k(f_2(x_1, x_2), f_2(y_1, y_2))\}\\
    &\le \max\{d_k(x_1, y_1), d_k(x_2, y_2)\},
\end{align*}
and hence, we get
\begin{align*}
    &\max\{d_{k^n}(f_1(x_1, x_2), f_1(y_1, y_2)), d_{k^n}(f_2(x_1, x_2), f_2(y_1, y_2))\}\\
    &\le \max\{d_{k^n}(x_1, y_1), d_{k^n}(x_2, y_2)\}.
\end{align*}
by Lemma \ref{lemma:power}.
Therefore, Lemma \ref{lemma:SP} implies that 
\begin{align*}
&d_{(k^n)^{\otimes 2}}(F(x_1, x_2), F(y_1, y_2))\\
&=\Bigl(d_{k^n}(f_1(x_1, x_2), f_1(y_1, y_2))^2+d_{k^n}(f_2(x_1, x_2), f_2(y_1, y_2))^2\\
    &\;\;\;\;\;\;\;\;-d_{k^n}(f_1(x_1, x_2), f_1(y_1, y_2))^2d_{k^n}(f_2(x_1, x_2), f_2(y_1, y_2))^2\Bigr)^{1/2}\\
&\le\sqrt{2}\Bigl(d_{k^n}(x_1, y_1)^2+d_{k^n}(x_2, y_2)^2-d_{k^n}(x_1, y_1)^2d_{k^n}(x_2, y_2)^2 \Bigr)^{1/2}\;\;\;(\mbox{by Lemma \ref{lemma:SP}})\\
&=\sqrt{2}d_{(k^n)^{\otimes 2}}((x_1, x_2), (y_1, y_2)).
\end{align*}
holds for all $(x_1, x_2), (y_1, y_2)\in X\times X$.
\end{proof}

\begin{cor}[$\mbox{\cite[Theorem 4.1]{Set}}$]
If $F:\D^2\rightarrow\D^2$ is a holomorphic map on $\D^2$, then $F$ satisfies
\begin{equation*}
\rho(F(z_1, z_2), F(w_1, w_2))\le \sqrt{2}\rho((z_1, z_2), (w_1, w_2))
\end{equation*}
for all $(z_1, z_2), (w_1, w_2)\in\D^2$.
\end{cor}
\begin{proof}
This immediately follows from Example \ref{example:3-1} and Theorem \ref{theorem:1}.
\end{proof}

\begin{remark}\label{remark:3-4}
\upshape In \cite[Example 3.1 and Remark 3.4]{DKS}, we found $\zeta$, $\lambda$, ${\bf z}$, ${\bf w}\in \D^2$ such that
\begin{equation}\label{equation:3-2}
\rho(\zeta, \lambda)\le \rho({\bf z},{\bf w})
\end{equation}
holds, but there never exists any holomorphic map $F:\D^2\rightarrow\D^2$ that satisfies $F({\bf z})=\zeta$, $F({\bf w})=\lambda$. On the other hand, we proved a suitable interpolation theorem for points in $\D^2$ that satisfy inequality (\ref{equation:3-2}) (see \cite[Theorem 1.2]{DKS}).
\end{remark}

\section*{Acknowledgment}
The author acknowledges Professor Michio Seto and Mr Deepak K. D. for fruitful discussions in our joint work \cite{DKS} (especially, concerning Remark \ref{remark:3-4}).
Moreover, Professor Seto suggested that the author study the Bergman space as well as the Hardy space in Theorem \ref{theorem:1}.
The author is grateful to Professor Yoshimichi Ueda for his comments on the draft of this paper. This work was supported by JSPS Research Fellowship for Young Scientists (KAKENHI Grant Number JP 23KJ1070).


\end{document}